\documentclass[10pt,twoside]{siamltex}
\usepackage{setspace}
\usepackage{color}
\usepackage{amsfonts,epsfig}
\usepackage{xy}
\input{xy}
\xyoption{all}
\newtheorem{remark}[theorem]{Remark}

\newtheorem{example}[theorem]{Example}

\begin{document}

%  Leave these commented lines here
% \input{elaheader-volx-xx.tex}
% \setcounter{page}{1}

% \renewcommand{\thefootnote}{\fnsymbol{footnote}}
% \renewcommand{\thefootnote}{\arabic{footnote}}
% \renewcommand{\theequation}{\thesection.\arabic{equation}}

\bibliographystyle{plain}
\title{
A parametrization of matrix conjugacy orbit sets as unions of
affine planes }
% Leave blank; editors will write the exact dates above

\author{
Peteris Daugulis\thanks{Department of Mathematics, Daugavpils
University, Daugavpils, Parades 1, Latvia
(peteris.daugulis@du.lv).}
% Remember to put \and between any two authors
%\and
%Danny Hershkowitz\thanks{Mathematics Department, Technion,
%Haifa 32000, Israel (hershkow@math.technion.ac.il, \newline
%goldberg@math.technion.ac.il).}
%\and
%Moshe Goldberg\footnotemark[3]
}
% Note that \footnotemark[3]} is used for the third author
% because of the same affiliation for the second and third authors.
% If the same affiliation is to be used for the first and second authors,
% \footnotemark[2] should be used instead of \thanks{} for the second author.

% Authors and running title to go on top of each page
\pagestyle{myheadings} \markboth{P.Daugulis}{A note on
generalization of eigenvector centrality for bipartite graphs and
applications} \maketitle

\begin{abstract}
%% Text of abstract
The problem of finding a canonical form of complex matrices up to
conjugacy with the set of canonical matrices being a union of
affine planes in the matrix space is considered. A solution of the
problem is given producing a new canonical form for matrices over
algebraically closed fields.
\end{abstract}

%\begin{AMS}
%15A18, 15A48, 90B10.
%\end{AMS}

%%%%%%%%%%%%%%%%%%%%%%%%%%%%%%%%%%%%%%%%%%%%%%%%%%%%%%%%%%%%%

\section{Introduction} \label{intro-sec}
It is known that sets of representatives of matrix conjugacy
classes in Jordan or rational canonical (elementary divisors
version) forms in general do not constitute algebraic varieties in
the ambient matrix spaces.

\begin{example} Consider $2\times 2$ matrices over $\mathbb{C}$.
A set of the Jordan forms can be chosen as $J_{1}\cup J_{2}$ where

\begin{eqnarray*}J_{1}=
\{\left[%
\begin{array}{cc}
  a &  \\
   & b \\
\end{array}%
\right],\forall\ a,b\in \mathbb{C}\}, J_{2}=\{\left[%
\begin{array}{cc}
  c &  \\
  1 & c \\
\end{array}%
\right], \forall\ c\in \mathbb{C}\}.
\end{eqnarray*}
If $a\neq b$ then matrices $\left[%
\begin{array}{cc}
  a &  \\
   & b \\
\end{array}%
\right]$ and $\left[%
\begin{array}{cc}
  b &  \\
   & a \\
\end{array}%
\right]$ are conjugate therefore $J_{1}$ contains two
representatives of each matrix conjugacy class with two distinct
eigenvalues.
\end{example}

M.L.Kontsevich \cite{A}, p.127, has proposed the problem of
finding a canonical form of matrices over $\mathbb{C}$ for which
the set of representatives of matrix conjugacy classes would be a
union of disjoint affine planes - cosets in the ambient
$\mathbb{C}$-linear matrix space modulo a subspace. In other
words, for each $n\ge 2\ $ find a union of disjoint affine planes
$\mathcal{A}$ in the $\mathbb{C}$-linear space
$\mathcal{M}at(n,n,\mathbb{C})\simeq \mathbb{C}^{n^2}$ so that
each matrix conjugacy orbit intersects $\mathcal{A}$ in exactly
one point. The problem may be motivated by a desire to provide
more links between the representation theory of algebras and
algebraic geometry.

Note that neither the Jordan form nor the rational canonical forms
 satisfy this condition. The rational canonical form with
invariant factors may be investigated if one wants to use affine
varieties of higher degrees but this is not within the scope of
this paper.

In this paper we present a solution of the above problem by
developing a new canonical form of matrices which we call \sl
canonical plane form.\rm\ Matrices considered are defined over
$\mathbb{C}$. Similar results can be proved for any algebraically
closed field.

\section{Description of canonical plane matrices and affine planes}

The goal of this section is to describe complex matrices called
\sl canonical plane matrices\rm\ which will be interpreted as
points of affine planes in the $\mathbb{C}$-linear matrix spaces.
In the first step we define matrices over multivariate polynomial
rings. Then we obtain complex matrices using substitutions.

Matrices over polynomial rings are denoted using bold letters.
$\textbf{E}_{m}$ denotes the identity matrix.

\subsection{Matrices over multivariate polynomial rings}

\subsubsection{Partitions}
Given $n\in \mathbb{N}, n\ge 2$, we consider all
 (nonincreasing) partitions of $n$, i.e. sequences of natural
numbers  - addends $(n_{1},...,n_{d})$ such that $n_{i}\ge
n_{i+1}$, $\forall\ i:\ 1\le i\le d-1$ and
$\sum\limits_{i=1}^{d}n_{i}=n$. Maximal constant contiguous
subsequences, i.e. maximal subsequences $(n_{r},...,n_{r+s})$ with
$n_{r}=...=n_{r+s}$ will be called \sl stacks.\rm\ A partition of
$n$ having $t$ stacks of lengths $l_{1},...,l_{t}$ with distinct
addends $m_{1}>m_{2}>...>m_{t}$ may be denoted, symbolically, as
$(m_{1}^{l_{1}},...,m_{t}^{l_{t}})$.
\subsubsection{Diagonal matrices}\label{212} Given a partition
$\pi=(n_{1},n_{2},...,n_{d})$ of $n$ define a diagonal matrix
\begin{eqnarray*} \textbf{D}_{\pi}=\bigoplus_{i=1}^{d}X_{i}\cdot
\textbf{E}_{n_{i}}\in
\mathcal{M}at(n,n,\mathbb{C}[X_{1},...,X_{d}]).
\end{eqnarray*}
We use the convention that indices of the variables in diagonal
blocks are arranged in the increasing order starting from the
upper left corner as shown in the example.

\begin{example} Let $n=8$, $\pi=(3,3,2)$, then $$ \textbf{D}_{\pi}=X_{1}\textbf{E}_{3}\oplus X_{2}\textbf{E}_{3}\oplus X_{3}\textbf{E}_{2}=\footnotesize\left[%
\begin{array}{ccc|ccc|cc}
  X_{1} &  &  &  &  &  &  &  \\
   &  X_{1}&  &  &  &  &  &  \\
   &  & X_{1} &  &  &  &  &  \\
  \hline
   &  &  &  X_{2}&  &  &  &  \\
   &  &  &  &  X_{2}&  &  &  \\
   &  & &  &  &  X_{2}&  &  \\
  \hline
   &  &  &  &  &  & X_{3} &  \\
   &  &  &  &  & &  & X_{3} \\
\end{array}%
\right]\normalsize .
$$\rm
\end{example}
\subsubsection{Generalized companion matrices}\rm

For any $l\le d$ define a $l\times l$ matrix
$\textbf{R}_{l}(X_{i_{1}},...,X_{i_{l}})$ as follows:
\begin{eqnarray*}
\textbf{R}_{l}(X_{i_{1}},...,X_{i_{l}})=
\left[%
\begin{array}{cccc|c}
  0& 0& ...& 0&(-1)^{l+1}e_{l}(X_{i_{1}},...,X_{i_{l}})\\
 1  & 0 &  ...&0&(-1)^{l}e_{l-1}(X_{i_{1}},...,X_{i_{l}})\\
  0 &  1  &...& 0&\\
  ... & ...& ... &...&...\\
   0& 0&  ... & 1&e_{1}(X_{i_{1}},...,X_{i_{l}})
\end{array}%
\right]
\end{eqnarray*}
where $e_{i}$ is the elementary symmetric polynomial of degree $i$
in $l$ variables. It is the
 $l\times l$ companion matrix having $l$ distinct "eigenvalues"
$X_{i_{1}},...,X_{i_{l}}$.

%In other description
%
%$$\textbf{R}_{l}(X_{i_{1}},...,X_{i_{l}})=
%\left[%
%\begin{array}{c|c}
%  \textbf{O}_{1,l-1} & \textbf{s}_{l} \\
%  \hhline{-|~}
%  \textbf{E}_{l} &  \\
%\end{array}%
%\right]
%$$
%where $\textbf{s}_{l}=\left[%
%\begin{array}{c}
%   (-1)^{l+1}e_{l}(X_{i_{1}},...,X_{i_{l}})\\
%  ... \\
%  e_{1}(X_{i_{1}},...,X_{i_{l}}) \\
%\end{array}%
%\right]$.

\begin{example}$\textbf{R}_{3}(X_{1},X_{2},X_{3})=\left[%
\begin{array}{cc|c}
  0 & 0 & X_{1}X_{2}X_{3} \\
  1 & 0 & -X_{1}X_{2}-X_{1}X_{3}-X_{2}X_{3} \\
  0 & 1 & X_{1}+X_{2}+X_{3} \\
\end{array}%
\right]$.
\end{example}
\newpage

\subsubsection{Block lower triangular matrices}\rm

Given a partition of $\pi=(n_{1},...,n_{d})$ of $n$ define a lower
block triangular matrix $\textbf{P}_{\pi}$ by modifying
$\textbf{D}_{\pi}$ in two steps as follows:
\begin{enumerate}
    \item for each pair $(n_{j},n_{j+1})$ with $n_{j}>n_{j+1}$ insert
    the
    mat\-rix

    $\textbf{F}_{n_{j+1},n_{j}}=[\textbf{O}_{n_{j+1},n_{j}-n_{j+1}}|\textbf{E}_{n_{j+1}}]$ as shown:
\begin{eqnarray*}
\left[%
\begin{array}{c|c}
  X_{j}\cdot \textbf{E}_{n_{j}} &  \\
  \hline
   & X_{j+1}\cdot \textbf{E}_{n_{j+1}} \\
\end{array}%
\right] \longrightarrow
\left[%
\begin{array}{c|c}
  X_{j}\cdot \textbf{E}_{n_{j}} &  \\
  \hline
  \textbf{F}_{n_{j+1},n_{j}} & X_{j+1}\cdot \textbf{E}_{n_{j+1}} \\
\end{array}%
\right],
\end{eqnarray*}
\bigskip
    \item for each stack $(n_{r},n_{r+1},...,n_{r+u})$ of $\pi$ with
    $n_{r}=...=n_{r}=m$ substitute the submatrix
\begin{eqnarray*}
\left[%
\begin{array}{c|c|c|c|c}
  X_{r}\cdot \textbf{E}_{m} & & & &\\
  \hline
   & X_{r+1}\cdot \textbf{E}_{m} & & &\\
  \hline
   &   & X_{r+2}\cdot \textbf{E}_{m} & &\\
   \hline
   & &  & ... &\\
   \hline
   & & & & X_{r+u}\cdot \textbf{E}_{m}
\end{array}%
\right]\end{eqnarray*} by
$$
\textbf{R}_{u+1}(X_{r},...,X_{r+u})\otimes \textbf{E}_{m}=$$
\begin{eqnarray*}
\left[%
\begin{array}{c|c|c|c|c}
  & & & & (-1)^{u+2}e_{u+1}(X_{r},...,X_{r+u})\cdot \textbf{E}_{m}\\
  \hline
\textbf{E}_{m}  &  & & &(-1)^{u+1}e_{u}(X_{r},...,X_{r+u})\cdot \textbf{E}_{m}\\
  \hline
   &  \textbf{E}_{m} &  & &\\
   \hline
   & &  & ... &\\
   \hline
   & & & \textbf{E}_{m} & e_{1}(X_{r},...,X_{r+u})\cdot \textbf{E}_{m}
\end{array}%
\right]
\end{eqnarray*}
\end{enumerate}
Thus stacks of $\pi$ bijectively correspond to diagonal blocks of
$\textbf{P}_{\pi}$ and off-diagonal blocks of $\textbf{P}_{\pi}$
are identity matrices coming from $\textbf{F}_{n_{j+1},n_{j}}$.
%
%More precisely, denote by $\textbf{E}_{m}(i,j)$ the result of the
%insertion of the identity matrix $\textbf{E}_{m}$ as a block into
%the zero matrix with its upper right corner at $(i,j)$. Suppose
%$\pi$ has $t$ stacks of length $l_{1},...,l_{t}$ with distinct
%addends $m_{1}>m_{2}>...>m_{t}$. Denote $l_{i}m_{i}$ by $v_{i}$,
%$\sum_{j=1}^{i}v_{j}$ by $z_{i}$. Then
%
%\begin{equation}
%\textbf{A}_{\pi}=\bigoplus_{i=1}^{t}\Big(\textbf{R}_{l_{i}}\otimes
%\textbf{E}_{m_{i}}\Big)+\sum_{i=2}^{t}\textbf{E}_{m_{i}}(z_{i}+1,z_{i}).
%\end{equation}
%$\textbf{A}_{\pi}$ has $t$ diagonal blocks of sizes
%$l_{1}m_{1}\times l_{1}m_{1}$,...,$l_{t}m_{t}\times l_{t}m_{t}$,
%the off-diagonal blocks are $t-1$ identity matrices
%$\textbf{E}_{m_{2}}$,...,$\textbf{E}_{m_{t}}$.

\begin{example} Let $n=8$, $\pi=(3,3,2)$, then $$ \textbf{P}_{\pi}=\footnotesize\left[%
\begin{array}{cccccc|cc}
   &  &  & -X_{1}X_{2} &  &  &  &  \\
   &  &  &  & -X_{1}X_{2} &  &  &  \\
   &  &  &  &  & -X_{1}X_{2} &  &  \\
  1 &  &  & X_{1}+X_{2} &  &  &  &  \\
   & 1 &  &  & X_{1}+X_{2} &  &  &  \\
   &  & 1 &  &  & X_{1}+X_{2} &  &  \\
  \hline
   &  &  &  & 1 &  & X_{3} &  \\
   &  &  &  &  & 1 &  & X_{3} \\
\end{array}%
\right]\normalsize .
$$\rm
\end{example}

\subsection{Complex canonical plane matrices and affine planes}

\subsubsection{Substitutions maps}

 %The substitution map for polynomials
%$$\varphi:\mathbb{C}[X_{1},...,X_{d}]\times
%\mathbb{C}^{d}\rightarrow \mathbb{C},$$ $$
%(f,(a_{1},...,a_{d}))\mapsto f(a_{1},...,a_{d})$$ induces the

Define the substitution map
$$\varphi:\mathcal{M}at(n,n,\mathbb{C}[X_{1},...,X_{d}])\times
\mathbb{C}^{d} \rightarrow \mathcal{M}at(n,n,\mathbb{C}),$$
$$
\Big([f_{ij}],(a_{1},...,a_{d})\Big)\mapsto
[f_{ij}(a_{1},...,a_{d})].
$$
%$$
%\left[%
%\begin{array}{ccc}
%  f_{11} & ... & f_{1n} \\
%  ... & ... & ... \\
%  f_{n1} & ... & f_{nn} \\
%\end{array}%
%\right],a_{1},...,a_{d}\mapsto \left[%
%\begin{array}{ccc}
%  f_{11}(a_{1},...,a_{d}) & ... & f_{1n}(a_{1},...,a_{d}) \\
%  ... & ... & ... \\
%  f_{n1}(a_{1},...,a_{d}) & ... & f_{nn}(a_{1},...,a_{d}) \\
%\end{array}%
%\right]
%$$
%Given
%$\textbf{A}_{\pi}\in\mathcal{M}at(n,n,\mathbb{C}[X_{1},...,X_{d}])$
%and $a=(a_{1},...,a_{d})\in \mathbb{C}^{d}$ we denote
%$\varphi(\textbf{A}_{\pi},a)$ by $A_{\pi}(a)$.
Denote the image of $\textbf{P}_{\pi}\times \mathbb{C}^{d}$ under
$\varphi$ by $\mathcal{P}_{\pi}$.

\subsubsection{Description of canonical plane matrices} Let $f=X^{l}-\sum\limits_{i=0}^{l-1}a_{i}X^{i}\in
\mathbb{C}[X]$. Define a companion matrix $C_{f}$ of $f$ in a
standard form:
\begin{eqnarray*}
C_{f}=C(a_{0},...,a_{l-1})=\left[%
\begin{array}{cccc|c}
  0& 0& ...& 0&a_{0}\\
 1  & 0 &  ...&0&a_{1}\\
  0 &  1  &...& &\\
  ... & ...& ... &...&...\\
   0& 0&  ... & 1&a_{l-1}
\end{array}%
\right]
\end{eqnarray*}
Each $A\in \mathcal{P}_{\pi}$ is a block lower triangular matrix
with the diagonal blocks in form $C(a_{0},...,a_{l-1})\otimes
\textbf{E}_{m}$ for some $l$ and $(a_{0},...,a_{l-1})\in
\mathbb{C}^{l}$.

\begin{example}\label{run}
Let $n=8$, $\pi=(3,3,2)$, then $$A=
\scriptsize\left[%
\begin{array}{c|c|c|c|c|c|c|c}
   &  &  & -2 &  &  &  &  \\
   \hline
   &  &  &  & -2 &  &  &  \\
   \hline
   &  &  &  &  & -2 &  &  \\
   \hline
  1 &  &  & 3 &  &  &  &  \\
  \hline
   & 1 &  &  & 3 &  &  &  \\
   \hline
   &  & 1 &  &  & 3 &  &  \\
  \hline
   &  &  &  & 1 &  & 1 &  \\
   \hline
   &  &  &  &  & 1 &  & 1 \\
\end{array}%
\right]\normalsize=\varphi(\textbf{P}_{\pi},(1,2,1)).
$$
\end{example}
\rm

\subsubsection{Inverse image of a canonical plane matrix} Given $A\in
\mathcal{P}_{\pi}$ we can find an element of its inverse image
under $\varphi$ as follows. For each diagonal block
$C(a_{0},...,a_{l-1})\otimes \textbf{E}_{m}$ of $A$ solve the
equation $X^{l}-\sum\limits_{i=0}^{l-1}a_{i}X^{i}=0$, get the
multiset of roots $[\lambda_{1},...,\lambda_{l}]$, its elements
are the complex numbers which must be substituted (in any order)
into the polynomial arguments in the corresponding diagonal block
of $\textbf{P}_{\pi}$.
 Note that the vector $(a_{0},...,a_{l-1})$
uniquely determines the multiset of roots
$[\lambda_{1},...,\lambda_{l}]$ and vice versa.

\subsubsection{Canonical affine planes}

Affine planes in a linear space $L$ are identified with additive
cosets modulo a subspace $V$ of $L$. Thus to define an affine
plane $\mathcal{P}\subseteq L$ we need to fix one element $l\in
\mathcal{P}$ (a constant shift) and describe $V$ (linear part).
Dimension of $\mathcal{P}$ is equal to $\dim V$.
\begin{theorem} \label{aff-th} $n\in \mathbb{N}$, $n\ge 2$, $\pi=(n_{1},...,n_{d})$
- a partition of $n$. Then $\mathcal{P}_{\pi}$
 is an affine
plane in $\mathcal{M}at(n,n,\mathbb{C})$ of dimension $d$.
\end{theorem}
\begin{proof} We will express $A\in\mathcal{P}_{\pi}$ as a sum
\begin{eqnarray*}
A=S+L(A)
\end{eqnarray*}
 of a constant matrix $S$ (the shift) and a variable
matrix $L(A)$ (the linear part) as follows.  The nonzero elements
of $S$ are the off-diagonal blocks together with the identity
submatrices of generalized companion matrices. Thus $S$ is
uniquely defined for all matrices in $\mathcal{P}_{\pi}$. We
define $L(A)=A-S$ and observe that the possibly nonzero elements
of $L(A)$ are the blocks of form $[a_{0}|...|a_{l-1}]^{T}\otimes
\textbf{E}_{m}$ corresponding to last columns of generalized
companion matrices.

%We have to show that for eac
%
%For each diagonal block of $A\in \mathcal{A}_{\pi}$ having the
%block column $C(a_{0},...,a_{l-1})\otimes \textbf{E}_{m}$ one can
%find complex numbers $c_{1},...,c_{l}$ which are the roots of the
%equation $X^{l}-\sum_{i=0}^{l-1}a_{i}X^{i}=0$. Substituting
%$c_{i}$ for $X_{i}$ Thus for each stack of $\pi$ the entries of
%the coresponding $R$ matrix can be

Supppose $A\in \mathcal{P}_{\pi}$ has the $t$ diagonal blocks
$C(a_{j0},...,a_{j,l_{j}-1})\otimes \textbf{E}_{m_{j}}$, $\forall
j\in \{1,...,t\}$. For each $j\in \{1,...,t\}$ solve the \sl
$j$-th stack equation\rm\
$X^{l_{j}}-\sum\limits_{i=0}^{l_{j}-1}a_{ji}X^{i}=0$,
 get the multiset of
roots $[\lambda_{j1},...,\lambda_{j,l_{j}}]$, $\forall j\in
\{1,...,t\}$.  We have that
\begin{eqnarray*}
A=\varphi_{\mathcal{M}}(\textbf{P}_{\pi},(\lambda_{11},...,\lambda_{1,l_{1}},...,\lambda_{t1},...,\lambda_{t,l_{t}})).
\end{eqnarray*}
Note that the ordering of the roots within each stack does not
change the value of $\varphi$. We see that $A\in
\mathcal{P}_{\pi}$ can be constructed for any sequence of vectors
$(a_{10},...,a_{1,l_{1}-1})$, ..., $(a_{t0},...,a_{t,l_{t}-1})$
thus $A-S$ runs through a linear subspace $V_{\pi}$ of
$\mathcal{M}at(n,n,\mathbb{C})$ as $A$ runs through
$\mathcal{P}_{\pi}$. We see that $\dim V_{\pi}$  is the sum of the
dimensions of vectors
$(a_{10},...,a_{1,l_{1}-1})$,...,$(a_{t0},...,a_{t,l_{t}-1})$
which is equal to $d$.
\end{proof}

\begin{example}
Consider the matrix $A$ of example \ref{run}. In this case
$$S=\footnotesize\left[%
\begin{array}{c|c|c|c|c|c|c|c}
   &  &  &  &  &  &  &  \\
   \hline
   &  &  &  &  &  &  &  \\
   \hline
   &  &  &  &  &  &  &  \\
   \hline
  1 &  &  &  &  &  &  &  \\
  \hline
   & 1 &  &  &  &  &  &  \\
   \hline
   &  & 1 &  &  &  &  &  \\
  \hline
   &  &  &  & 1 &  &  &  \\
   \hline
   &  &  &  &  & 1 &  &  \\
\end{array}%
\right]\normalsize,L(A)=\footnotesize\left[%
\begin{array}{c|c|c|c|c|c|c|c}
   &  &  & -2 &  &  &  &  \\
   \hline
   &  &  &  & -2 &  &  &  \\
   \hline
   &  &  &  &  & -2 &  &  \\
   \hline
   &  &  & 3 &  &  &  &  \\
  \hline
   &  &  &  & 3 &  &  &  \\
   \hline
   &  &  &  &  & 3 &  &  \\
  \hline
   &  &  &  &  &  & 1 &  \\
   \hline
   &  &  &  &  &  &  & 1 \\
\end{array}%
\right]\normalsize.
$$
$\dim V_{\pi}=3$, $V_{\pi}=\langle V_{1},V_{2},V_{3}\rangle$ where
$$
V_{1}=\footnotesize\left[%
\begin{array}{c|c|c|c|c|c|c|c}
   &  &  & 1 &  &  &  &  \\
   \hline
   &  &  &  & 1 &  &  &  \\
   \hline
   &  &  &  &  & 1 &  &  \\
   \hline
   &  &  &  &  &  &  &  \\
  \hline
   &  &  &  &  &  &  &  \\
   \hline
   &  &  &  &  &  &  &  \\
  \hline
   &  &  &  &  &  &  &  \\
   \hline
   &  &  &  &  &  &  &  \\
\end{array}%
\right]\normalsize, V_{2}=\footnotesize\left[%
\begin{array}{c|c|c|c|c|c|c|c}
   &  &  &  &  &  &  &  \\
   \hline
   &  &  &  &  &  &  &  \\
   \hline
   &  &  &  &  &  &  &  \\
   \hline
   &  &  & 1 &  &  &  &  \\
  \hline
   &  &  &  & 1 &  &  &  \\
   \hline
   &  &  &  &  & 1 &  &  \\
  \hline
   &  &  &  &  &  &  &  \\
   \hline
   &  &  &  &  &  &  &  \\
\end{array}%
\right]\normalsize,$$ $$V_{3}=\footnotesize\left[%
\begin{array}{c|c|c|c|c|c|c|c}
   &  &  &  &  &  &  &  \\
   \hline
   &  &  &  &  &  &  &  \\
   \hline
   &  &  &  &  &  &  &  \\
   \hline
   &  &  &  &  &  &  &  \\
  \hline
   &  &  &  &  &  &  &  \\
   \hline
   &  &  &  &  &  &  &  \\
  \hline
   &  &  &  &  &  & 1 &  \\
   \hline
   &  &  &  &  &  &  & 1 \\
\end{array}%
\right]\normalsize.
$$
\end{example}
\begin{theorem} \label{aff-th2} $n\in \mathbb{N}$, $n\ge 2$, $\pi$
and $\pi'$ - distinct partitions of $n$. Then
\begin{eqnarray*}
\mathcal{P}_{\pi}\cap \mathcal{P}_{\pi'}=\emptyset.
\end{eqnarray*}
\end{theorem}
\begin{proof} We show that supports of constant shift matrices are different. Given $\pi\neq \pi'$ with $\pi=(n_{1},n_{2},...)$ and
$\pi'=(n_{1}',n_{2}',...)$ consider the distinct pair $n_{j}\neq
n_{j}'$ with the minimal $j$. There are two possibilities: $(1)$
new stacks start from $n_{j}$ and $n_{j}'$ or $(2)$ a new stack
starts from one of $n_{j}$ or $n_{j}'$. In the first case the
off-diagonal blocks of $\textbf{P}_{\pi}$ and $\textbf{P}_{\pi'}$
inserted to the left of diagonal blocks for the new stacks have
$1$'s in different positions. In the second case in
$\textbf{P}_{\pi}$ and $\textbf{P}_{\pi'}$ we have diagonal blocks
of distinct sizes having the same position of upper left corners,
by considering identity matrix blocks of generalized companion
matrices it follows that at least one $1$' in the off-diagonal
block below the shorter diagonal block in, say, $\textbf{P}_{\pi}$
is absent in $\textbf{P}_{\pi'}$.
\end{proof}
\newpage

\section{Main results}
\subsection{Jordan forms of canonical plane matrices}\
\subsubsection{Preparations}\label{311}
%We introduce useful notations similar to those used in the proof
%of theorem \ref{aff-th}.

Let a partition $\pi$ have $t$ stacks of lengths $l_{1},...,l_{t}$
with distinct addends $m_{1}>m_{2}>...>m_{t}$. Define $m_{t+1}=0$.
Define $s_{j}=\sum\limits_{i=1}^{j}l_{i}$. Define
$P_{j}=\{m_{1}-m_{j}+1,...,m_{1}-m_{j+1}\}$.

Let $A\in \mathcal{P}_{\pi}$ have diagonal blocks
$C(a_{j0},...a_{j,l_{j}-1})\otimes \textbf{E}_{m_{j}}$, $j\in
\{1,...,t\}$. For each $j\in \{1,...,t\}$ solve the $j$th stack
equation

\begin{eqnarray*} X^{l_{j}}-\sum_{i=0}^{l_{j}-1}a_{ji}X^{i}=0,
\end{eqnarray*}
 get the multiset  of roots
$[\lambda_{j1},...,\lambda_{j,l_{j}}]$,$j\in \{1,...,t\}$. Denote
the muliplicity of $\lambda$ as a root for the $i$th stack
equation by $\mu(\lambda,i)$. Denote
$\sum\limits_{i=1}^{j}\mu(\lambda,i)$ by $\alpha(\lambda,j)$.

We think of $A$ acting in $\mathbb{C}^{n}$ - the $n\times 1$
column space with the standart basis
$\mathcal{B}=\{e_{1},...,e_{n}\}$, $e_{i}=E_{i1}$ where $E_{ij}$
is a matrix unit.

The goal of this subsection is to find the Jordan form of $A\in
\mathcal{P}_{\pi}$. This is done in several steps:
\begin{enumerate}
\item by inspecting the digraph of $A$ we find a decomposition of
$\mathbb{C}^{n}$ into a direct sum of $m_{1}$ $A$-invariant
subspaces which are generated by $e_{1}$,...,$e_{m_{1}}$ as
$\mathbb{C}[A]$-modules,
    \item decompose $A$ into a direct sum by permutations,
    \item decompose the obtained direct summands of $A$ into the Jordan form using their block
    structure.
\end{enumerate}

\subsubsection{Invariant subspaces} \label{312} For each $i\in \{1,...,m_{1}\}$
define a $A$-invariant subspace $V_{i}=\mathbb{C}[A]\cdot e_{i}$.
For each $j\in \{1,...,t\}$ define a $s_{j}\times s_{j}$ matrix
\begin{eqnarray*}
G_{j}=\bigoplus_{i=1}^{j}C(c_{i0},...,c_{i,l_{i}-1})+\sum_{i=1}^{j-1}E_{s_{i}+1,s_{i}}.
\end{eqnarray*}
Here we also use the convention about the block ordering as in
\ref{212}, see example \ref{run2} below.

\begin{theorem}\label{31} In the notations of \ref{311} we have
\begin{enumerate}
    \item $\mathbb{C}^{n}=\bigoplus\limits_{i=1}^{m_{1}}V_{i}$,
    \item $\dim V_{i}=s_{j}$ for $i\in P_{j}$,
    \item If $i\in P_{j}$ then the restriction of $A$ to $V_{i}$ with respect to
    $\mathcal{B}$ is $G_{j}$.
\end{enumerate}
\end{theorem}
\begin{proof}
The statements are proved by considering the images of
$e_{1},...,e_{m_{1}}$ under powers of $A$, using the digraph of
$A$ and induction. The digraph of $A$ with respect to
$\mathcal{B}$ decomposes into $m_{1}$ weakly connected components
corresponding to $V_{1},...,V_{m_{1}}$. Vertex sets of these
components and restrictions of $A$ can be explicitly described to
prove the restriction statement. The number of elements of
$\mathcal{B}$ in the component corresponding to $V_{i}$ is equal
to $s_{j}$ if $i\in P_{j}$ which implies the dimension statement.
Further details are omitted.
\end{proof}

\begin{example}\label{run2} Consider the matrix $A$ of example \ref{run}. Its digraph
is shown in Fig.1.
$$
\xymatrix{ e_{1}\ar[r]&e_{4}\ar@(ul,ur)^{3}\ar@/_/[l]_{-2}&\\
e_{2}\ar[r]&e_{5}\ar[r]\ar@(ul,ur)^{3}\ar@/_/[l]_{-2}&e_{7}\ar@(ul,ur)\\
e_{3}\ar[r]&e_{6}\ar[r]\ar@(ul,ur)^{3}\ar@/_/[l]_{-2}&
e_{8}\ar@(ul,ur)}
$$
\begin{center}
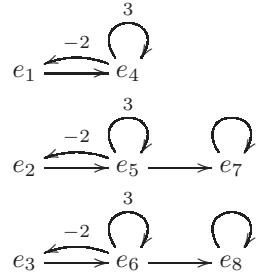
Fig.1.  - the digraph for $A$ of example \ref{run}

\end{center}
In this case $\mathbb{C}^{8}=V_{1}\oplus V_{2}\oplus V_{3}$ where
$V_{1}=\langle e_{1},e_{4}\rangle$, $V_{2}=\langle
e_{2},e_{5},e_{7}\rangle$, $V_{3}=\langle
e_{3},e_{6},e_{8}\rangle$. $G_{1}=\left[%
\begin{array}{cc}
  0 & -2 \\
  1 & 3 \\
\end{array}%
\right]$, $G_{2}=\left[%
\begin{array}{cc|c}
  0 & -2 & 0 \\
  1 & 3 & 0 \\
  \hline
  0 & 1 & 1 \\
\end{array}%
\right]$.
\end{example}
\subsubsection{Decomposition by permutation}\
Denote the matrix conjugacy relation by $\simeq$.
\begin{theorem} \label{33} In the notations of \ref{311} we have
\begin{eqnarray*}
A\simeq \bigoplus_{j=1}^{t}(m_{j}-m_{j+1}) G_{j}.
\end{eqnarray*}
\end{theorem}
\begin{proof} The described direct sum is obtained by
permuting the rows and columns of $A$ following theorem \ref{31}.
\end{proof}

\begin{example} The matrix $A$ of example \ref{run} is permutation conjugate
to $$G_{1}\oplus G_{2} \oplus G_{2}=\scriptsize
\left[%
\begin{array}{cc|ccc|ccc}
   &  -2&  &  &  &  &  &  \\
   1& 3 &  &  &  &  &  &  \\
   \hline
   &  &  & -2 &  &  &  &  \\
   &  & 1 & 3 &  &  &  &  \\
   &  &  &  1& 1 &  &  &  \\
   \hline
   &  &  &  &  &  & -2 &  \\
   &  &  &  &  & 1 & 3 &  \\
   &  &  &  &  &  & 1 & 1 \\
\end{array}%
\right]\normalsize .
$$

\end{example}

\subsubsection{Jordan  forms}\
Denote the Jordan block with eigenvalue $\lambda$ and size $i$ by
$J_{i}(\lambda)$.
\begin{theorem}\label{35} In the notations of \ref{311} we have
\begin{eqnarray*}
G_{j}\simeq\bigoplus_{\lambda}J_{\alpha(\lambda,j)}(\lambda).
\end{eqnarray*}
\end{theorem}
\begin{proof} The result follows from the lower triangular block structure of $G_{j}$. $G_{j}$ is in block lower triangular form therefore its
characteristic polynomial $\chi(G_{j},x)$ is equal to the product
of the characteristic polynomials of diagonal blocks. It follows
that
$\chi(G_{j},x)=\prod\limits_{\lambda}(\lambda-x)^{\alpha(\lambda,j)}$.
The diagonal blocks of $G_{j}$ have subdiagonals with all elements
equal to $1$ and only zeros under subdiagonals.  The only nonzero
elements outside diagonal blocks of $G_{j}$ are $1$'s in the
corner positions between diagonal blocks.  It follows that $G_{j}$
has the global nonzero subdiagonal.  It follows that $G_{j}$ has a
cyclic vector $[1|0|0|...|0]^{T}$. Hence the minimal polynomial of
$G_{j}$ is $\pm\chi(G_{j},x)$. It follows $G_{j}$ has one Jordan
block for each eigenvalue.
\end{proof}
%\begin{remark}\rm It follows that $A$ has Jordan blocks with
%eigenvalue $\lambda$ of sizes $\mu(\lambda,1)$,
%$\underbrace{\mu(\lambda,1)+\mu(\lambda,2)}_{=\alpha(\lambda,2)}$,...,$\underbrace{\sum_{i=1}^{t}\mu(\lambda,i)}_{=\alpha(\lambda,t)}$
%only.
%\end{remark}
\begin{theorem} \label{37}In the notations of \ref{311} we have
\begin{equation}\label{dir_sum}
A\simeq\bigoplus_{j=1}^{t}(m_{j}-m_{j+1})\Big(
\bigoplus_{\lambda}J_{\alpha(\lambda,j))}(\lambda)\Big).
\end{equation}
\end{theorem}
\begin{proof} It follows from theorems \ref{33} and \ref{35}.
\end{proof}
\begin{example} The Jordan form of $A$ from example \ref{run} is
shown in Fig.2.
$$
\xymatrix{e_{1}'\ar@(ul,ur)^{2}&e_{4}'\ar@(ul,ur)&\\
e_{2}'\ar@(ul,ur)^{2}&e_{5}'\ar[r]\ar@(ul,ur)&e_{7}'\ar@(ul,ur)\\
e_{3}'\ar@(ul,ur)^{2}&e_{6}'\ar[r]\ar@(ul,ur)& e_{8}'\ar@(ul,ur)}
$$
\begin{center}
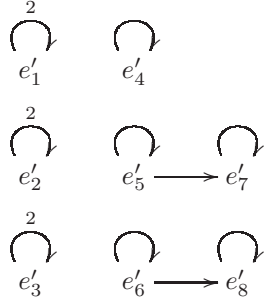
Fig.2. - the Jordan form digraph of $A$ from example \ref{run}
\end{center}
\end{example}
\subsection{Canonical plane forms of Jordan matrices}\
The goal of this subsection is to show that a direct sum of Jordan
matrices is conjugate to a canonical plane matrix.

\subsubsection{Weyr characteristic of canonical plane matrices}

The Jordan form of a matrix $B$ with the single eigenvalue
$\lambda$ is determined by its Weyr characteristic sequence
$\Omega(B,\lambda)=(\omega_{1},\omega_{2},...)$ where $\omega_{i}$
is the number of Jordan blocks of size at least $i$ (see
\cite{strang} for a discussion and recent applications).
\begin{example} If $B\simeq 2J_{1}(\lambda)\oplus J_{3}(\lambda)\oplus
J_{4}(\lambda)$ then $\Omega(B,\lambda)=(4,2,2,1,0,...)$,
$\omega_{u}=0$ for all $u\ge 5$.
\end{example}
The Jordan form of any matrix $B$ with eigenvalues
$\lambda_{1},\lambda_{2},...$ is determined by the Weyr array
$\Omega(B)=(\Omega(B,\lambda_{1}),\Omega(B,\lambda_{2}),...)$.

\begin{theorem} \label{3.9} In the notations of \ref{311} let an eigenvalue $\lambda$ of $A$ have nonzero
multiplicity $\mu(\lambda,j)$ if and only if $j\in
\{j_{1},j_{2},...,j_{v}\}$, $j_{1}<j_{2}<..<j_{v}$ Then
\begin{eqnarray*}
\Omega(A,\lambda)=(\underbrace{m_{j_{1}},...,m_{j_{1}}}_{\mu(\lambda,j_{1})\
copies}, \underbrace{m_{j_{2}},...,m_{j_{2}}}_{\mu(\lambda,j_{2})\
copies},...,\underbrace{m_{j_{v}},...,m_{j_{v}}}_{\mu(\lambda,j_{v})\
copies},0,...).
\end{eqnarray*}
\end{theorem}
\begin{proof} The statement follows from theorem \ref{37} by reordering the direct sum
(\ref{dir_sum}), counting the number of Jordan blocks of $A$ of
given $\lambda$ and size and using induction. We have that
\begin{eqnarray*}
A\simeq\bigoplus_{j=1}^{t}(m_{j}-m_{j+1})\Big(
\bigoplus_{\lambda}J_{\alpha(\lambda,j)}(\lambda)\Big)\simeq\\
\bigoplus_{\lambda}\Big(\bigoplus_{j=1}^{t}(m_{j}-m_{j+1})J_{\alpha(\lambda,j)}(\lambda)\Big)=\\
\bigoplus_{\lambda}\Big(\bigoplus_{j\in\{j_{1},j_{2},...,j_{v}\}
}(m_{j}-m_{j+1})J_{\alpha(\lambda,j)}(\lambda)\Big).
\end{eqnarray*}
Thus for a given $\lambda$
\begin{enumerate}
\item the minimal size of a Jordan block is
$\alpha(\lambda,j_{1})=\mu(\lambda,j_{1})$, there are $m_{j_{1}}$
Jordan blocks of sizes at least $1,...,\mu(\lambda,j_{1})$, start
$\Omega(A,\lambda)$ as
$\underbrace{(m_{j_{1}},...,m_{j_{1}})}_{\mu(\lambda,j_{1})\
copies}$, \item the next possible size of a Jordan block is
$\alpha(\lambda,j_{2})$, there are
$m_{j_{1}}-(m_{j_{1}}-m_{j_{2}})=m_{j_{2}}$ Jordan blocks of sizes
at least $\mu(\lambda,j_{1})+1,...,\alpha(\lambda,j_{2})$,
continue $\Omega(A,\lambda)$ as
$(\underbrace{m_{j_{1}},...,m_{j_{1}}}_{\mu(\lambda,j_{1})\
copies},\underbrace{m_{j_{2}},...,m_{j_{2}}}_{\mu(\lambda,j_{2})\
copies})$, \item ...
\end{enumerate}
\end{proof}
\subsubsection{From Jordan forms to canonical plane
matrices}\label{322}

We describe an algorithm for finding a canonical plane matrix
which is conjugate to a given matrix $B$.  Rearrange $\Omega(B)$
array as follows:
\begin{enumerate}
\item sort the nonzero entries of $\Omega(B)$ in nonincreasing
order, \item collect the equal nonzero entries into a sequence of
multisets
\begin{eqnarray*}
M(B)=(M_{1},M_{2},...,M_{t}),
\end{eqnarray*}
all elements of $M_{i}$ are equal to some $m_{i}$,
$m_{1}>m_{2}>...>m_{t}$.
\end{enumerate}
Denote $|M_{i}|$ by $l_{i}$, denote the multisets of eigenvalues
corresponding to the elements of $M_{i}$ by
$\Lambda_{i}=[\lambda_{i1},...,\lambda_{i,l_{i}}]$.  Define
$\pi=(m_{1}^{l_{1}},...,m_{t}^{l_{t}})$.
\begin{theorem}\label{8} In the above notations
\begin{eqnarray*}
B\simeq
\widetilde{B}=\varphi(\textbf{P}_{\pi},(\underbrace{\lambda_{11},...,\lambda_{1,l_{1}}}_{1st\
block},...,\underbrace{\lambda_{t1},...,\lambda_{t,l_{t}}}_{t-th\
block})).
\end{eqnarray*}
\end{theorem}
\begin{proof} Using the theorem \ref{3.9} we construct $\Omega(\widetilde{B})$ by considering diagonal blocks of $\widetilde{B}$:
\begin{enumerate}
\item from the $1$st block we get the initial elements equal to
$m_{1}$ of $\Omega(\widetilde{B},\lambda)$ for $\lambda\in
\Lambda_{1}$, thus get initial subsequences of
$\Omega(\widetilde{B},\lambda)$ of form
$\underbrace{(m_{1},...,m_{1})}_{\mu(\lambda,1)}$ for each
$\lambda$ with $\mu(\lambda,1)\neq 0$,

\item from the $2$nd block we get elements equal to $m_{2}$ of
$\Omega(\widetilde{B},\lambda)$ for $\lambda\in\Lambda_{2}$, thus
form subsequences of $\Omega(\widetilde{B},\lambda)$ of length
$\alpha(\lambda,2)$ by adding
$\underbrace{(m_{2},...,m_{2})}_{\mu(\lambda,2)\ times}$ for each
$\lambda$ with $\mu(\lambda,1)\neq 0$, $\mu(\lambda,2)\neq 0$ or
initiate $\Omega(\widetilde{B},\lambda)$ as
$\underbrace{(m_{2},...,m_{2})}_{\mu(\lambda,2)\ times}$ for each
$\lambda$ with $\mu(\lambda,1)=0$, $\mu(\lambda,2)\neq 0$. \item
...

\end{enumerate}
By induction on $|M|$ we see that
$\Omega(\widetilde{B})=\Omega(B)$ therefore $\widetilde{B}\simeq
B$.
\end{proof}
\begin{definition}\rm
For any matrix $B$ denote its canonical plane matrix
$\widetilde{B}$ by $\mathcal{A}ff(B)$.
\end{definition}
\begin{remark} Note that the canonical plane form of a matrix in
general is different from the Jordan form and the rational form
(either in invariant factors or elementary divisors version) as
example \ref{run} shows.
\end{remark}
\subsection{Bijectivity of the correspondence between canonical plane matrices and Jordan forms}\
\begin{theorem}\label{313}
\begin{enumerate}
\item $A$, $A'$ - canonical plane $n\times n$ matrices. Then
\begin{eqnarray*}
A\neq A'\Longrightarrow A\not\simeq A'.
\end{eqnarray*}

\item $B$,$B'$  - any $n\times n$ matrices. Then
\begin{eqnarray*}
B\not\simeq B'\Longrightarrow \mathcal{A}ff(B)\neq
\mathcal{A}ff(B').
\end{eqnarray*}
\end{enumerate}
\end{theorem}
\begin{proof}

1. Let $A\in\mathcal{P}_{\pi}$, $A'\in\mathcal{P}_{\pi'}$, $A\neq
A'$. If $\pi\neq \pi'$ then the sets of addends of $\pi$ and
$\pi'$ are different which by theorem \ref{3.9} implies
$\Omega(A)\neq \Omega(A')$. If $\pi=\pi'$ then $A=S+L(A)$,
$A'=S+L(A')$ as in the proof of theorem \ref{aff-th}. It follows
that $L(A)\neq L(A)'$ hence for at least one $\lambda\in
\mathbb{C}$ the multiplicity functions $\mu(\lambda,i)$ for $A$
and $A'$ are different. By theorem \ref{3.9} it folows that
$\Omega(A,\lambda)\neq \Omega(A',\lambda)$.

2. Let $B\not\simeq B'$. Then $\Omega(B)\neq \Omega(B')$. Hence
there exist $m$  and $\lambda\in \mathbb{C}$ such that the
multiplicities of $m$ in $\Omega(B,\lambda)$ and
$\Omega(B',\lambda)$ are different. Considering $M(B)$ and $M(B')$
as explained at the beginning of \ref{322} and theorem \ref{8} we
see that $L(B)\neq L(B')$. It follows that $ \mathcal{A}ff(B)\neq
\mathcal{A}ff(B')$.
\end{proof}
\subsection{Conclusion} For any $n\in \mathbb{N}$, $n\ge 2$, define
$\mathcal{P}=\bigcup\limits_{\pi}\mathcal{P}_{\pi},$ where the
union is taken over all partitions of $n$.

\begin{theorem}$\mathcal{P}$ contains each $n\times n$ matrix conjugacy class
exactly once.
\end{theorem}
\begin{proof} The statement follows from theorems \ref{8} and \ref{313}.
\end{proof}

\subsection{Appendix - canonical plane matrices for two values of $n$}

\subsubsection{$n=2$}

 There are two partitions of $2$: $\pi_{1}=(2)$ and $\pi_{2}=(1,1)$. The affine planes are
\begin{enumerate}

\item
   $\mathcal{P}_{\pi_{1}}=\{\left[%
\begin{array}{c|c}
  a &   \\
   \hline
     & a \\
\end{array}%
\right], \forall\ a\in \mathbb{C}\}$, each matrix in
$\mathcal{P}_{\pi_{1}}$ is in its Jordan form $2J_{1}(a)$,

\item $\mathcal{P}_{\pi_{2}}=\{\left[%
\begin{array}{c|c}
   &  a \\
   \hline
   1  & b \\
\end{array}%
\right], \forall\ a,b\in \mathbb{C}\}$, the Jordan form depends on
the roots of the characteristic polynomial $x^2-bx-a$:

\begin{enumerate}

\item if there are $2$ simple roots $\lambda_{1}$, $\lambda_{2}$
then $J_{1}(\lambda_{1})\oplus J_{1}(\lambda_{2})$, \item
 if there
is $1$ double root $\lambda_{0}$ then  $J_{2}(\lambda_{0})$.
\end{enumerate}

\end{enumerate}

\subsubsection{$n=5$}

We describe all canonical plane matrices for $n=5$ and give the
corresponding Jordan forms. One can check that there are $27$
subcases which correspond to the $27$ distinct Jordan forms for
$5\times 5$ complex matrices. Below for each subcase distinct
arguments for polynomial roots and Jordan blocks mean distinct
complex numbers.

There are seven partitions of $5$: $\pi_{1}=(5)$, $\pi_{2}=(4,1)$,
$\pi_{3}=(3,2)$, $\pi_{4}=(3,1,1)$, $\pi_{5}=(2,2,1)$,
$\pi_{6}=(2,1,1,1)$, $\pi_{7}=(1,1,1,1,1)$. For each partition
$\pi$ we give the description of $\mathcal{P}_{\pi}$:
\begin{enumerate}
    \item $\mathcal{P}_{\pi_{1}}=\footnotesize\{\left[%
\begin{array}{ccccc}
  a &  &  &  &  \\
   & a &  &  &  \\
   &  & a &  &  \\
   &  &  & a &  \\
   &  &  &  & a \\
\end{array}%
\right], \forall\ a\in \mathbb{C}\normalsize\},$ each matrix in
$\mathcal{P}_{\pi_{1}}$ is in its Jordan form $5J_{1}(a)$;
    \item $\mathcal{P}_{\pi_{2}}=\footnotesize\{\left[%
\begin{array}{cccc|c}
  a &  &  &  &  \\
   & a &  &  &  \\
   &  & a &  &  \\
   &  &  & a &  \\
   \hline
   &  &  & 1 & b \\
\end{array}%
\right], \forall\ a,b\in \mathbb{C}\normalsize\},$ the Jordan
forms:
\begin{enumerate}

\item $4J_{1}(a)\oplus J_{1}(b)$ if $a\ne b$, \item
 $3J_{1}(a)\oplus
J_{2}(a)$ if $a=b$;
\end{enumerate}
    \item $\mathcal{P}_{\pi_{3}}=\footnotesize\{\left[%
\begin{array}{ccc|cc}
  a &  &  &  &  \\
   & a &  &  &  \\
   &  & a &  &  \\
   \hline
   & 1 &  & b &  \\
   &  & 1 &  & b \\
\end{array}%
\right], \forall\ a,b\in \mathbb{C}\normalsize\},$ the Jordan
forms:

\begin{enumerate}

\item $3J_{1}(a)\oplus 2J_{1}(b)$ if $a\ne b$,

\item $J_{1}(a)\oplus 2J_{2}(a)$ if $a=b$;
\end{enumerate}

    \item $\mathcal{P}_{\pi_{4}}=\footnotesize\{\left[%
\begin{array}{ccc|cc}
  a &  &  &  &  \\
   & a &  &  &  \\
   &  & a &  &  \\
   \hline
   &  & 1 &  & b \\
   &  &  & 1 & c \\
\end{array}%
\right], \forall\ a,b,c\in \mathbb{C}\normalsize\},$ the Jordan
form depends on the roots of the polynomial $x^2-cx-b$ - the
characteristic polynomial of the block $
\footnotesize\left[%
\begin{array}{cc}
  & b \\
 1 & c \\
\end{array}%
\right] $:
\begin{enumerate}

\item $2$ simple roots $\lambda_{1}$, $\lambda_{2}$,
$\lambda_{i}\ne a$ -  $3J_{1}(a)\oplus J_{1}(\lambda_{1})\oplus
J_{1}(\lambda_{2})$,

\item $2$ simple roots $\lambda_{1}$, $\lambda_{2}$,
$\lambda_{1}=a$, $\lambda_{2}\ne a$ -  $2J_{1}(a)\oplus
J_{2}(a)\oplus J_{1}(\lambda_{2})$,

\item $1$ double root $\lambda_{1}\ne a$ - $3J_{1}(a)\oplus
J_{2}(\lambda_{1})$,

\item $1$ double root $\lambda_{1}= a$ - $2J_{1}(a)\oplus
J_{3}(a)$,
\end{enumerate}

\item $\mathcal{P}_{\pi_{5}}=\footnotesize\{\left[%
\begin{array}{cccc|c}
   &  & a &  &  \\
   &  &  & a &  \\
  1 &  & b &  &  \\
   & 1 &  & b &  \\
   \hline
   &  &  & 1 & c \\
\end{array}%
\right], \forall\ a,b,c\in \mathbb{C}\normalsize\},$ the Jordan
form depends on the roots of the polynomial $x^2-bx-a$ - the
characteristic polynomial of the block $
\footnotesize\left[%
\begin{array}{cc}
  & a \\
 1 & b \\
\end{array}%
\right] $:
\begin{enumerate}

\item $2$ simple roots $\lambda_{1}$, $\lambda_{2}$,
$\lambda_{i}\ne c$ -  $J_{1}(c)\oplus 2J_{1}(\lambda_{1})\oplus
2J_{1}(\lambda_{2})$,

\item $2$ simple roots $\lambda_{1}$, $\lambda_{2}$,
$\lambda_{1}=c$, $\lambda_{2}\ne c$ - $J_{1}(c)\oplus
J_{2}(c)\oplus 2J_{1}(\lambda_{2})$,

\item $1$ double root $\lambda_{1}\ne c$ -
$2J_{2}(\lambda_{1})\oplus J_{1}(c)$,

\item $1$ double root $\lambda_{1}= c$ - $J_{2}(c)\oplus
J_{3}(c)$,
\end{enumerate}

\item $\mathcal{P}_{\pi_{6}}=\footnotesize\{\left[%
\begin{array}{cc|ccc}
  a &  &  &  &  \\
   &  a&  &  &  \\
   \hline
   &  1&  &  &b  \\
   &  & 1 &  & c \\
   &  &  & 1 & d \\
\end{array}%
\right], \forall\ a,b,c,d\in \mathbb{C}\normalsize\},$ the Jordan
form depends on the roots of the polynomial $x^3-dx^2-cx-b$ - the
characteristic polynomial of the block $
\footnotesize\left[%
\begin{array}{ccc}
  & &b \\
  1& & c \\
  &1& d
\end{array}%
\right] $:
\begin{enumerate}

\item $3$ simple roots $\lambda_{1}$, $\lambda_{2}$,
$\lambda_{3},$ $\lambda_{i}\ne a$ - $2J_{1}(a)\oplus
J_{1}(\lambda_{1})\oplus J_{1}(\lambda_{2})\oplus
J_{1}(\lambda_{3})$,

\item $3$ simple roots $\lambda_{1}$, $\lambda_{2}$,
$\lambda_{3},$ $\lambda_{1}= a$, $\lambda_{2,3}\ne a$ -
$J_{1}(a)\oplus J_{2}(a)\oplus J_{1}(\lambda_{2})\oplus
J_{1}(\lambda_{3})$,

\item $1$ double root $\lambda_{1}\ne a$ and $1$ simple root
$\lambda_{3}\ne a$ - $2J_{1}(a)\oplus J_{2}(\lambda_{1})\oplus
J_{1}(\lambda_{3}),$

\item $1$ double root $\lambda_{1}= a$ and $1$ simple root
$\lambda_{3}\ne a$ - $J_{1}(a)\oplus J_{3}(a)\oplus
J_{1}(\lambda_{3}),$

\item $1$ double root $\lambda_{1}\ne a$ and $1$ simple root
$\lambda_{3}= a$ - $J_{1}(a)\oplus J_{2}(a)\oplus
J_{2}(\lambda_{1}),$

\item $1$ triple root $\lambda_{1}\ne a$ - $2J_{1}(a)\oplus
J_{3}(\lambda_{1})$,

\item $1$ triple root $\lambda_{1}= a$ - $J_{1}(a)\oplus
J_{4}(a)$;
\end{enumerate}

\item $\mathcal{P}_{\pi_{7}}=\footnotesize\{\left[%
\begin{array}{ccccc}
   &  &  &  & a \\
  1 &  &  &  & b \\
   & 1 &  &  &  c\\
   &  & 1 &  &  d\\
   &  &  & 1 & e \\
\end{array}%
\right], \forall\ a,b,c,d,e\in \mathbb{C}\normalsize\},$ the
Jordan form depends on the roots of the characteristic polynomial
of the matrix $x^5-ex^4-dx^3-cx^2-bx-a$:
\begin{enumerate}

\item $5$ simple roots $\lambda_{i}$, $i\in \{1,...,5\}$ -
$\bigoplus_{i=1}^{5} J_{1}(\lambda_{i})$,

\item $1$ double root $\lambda_{1}$  and $3$ simple roots
$\lambda_{i}$, $i\in \{3,...,5\}$ - $J_{2}(\lambda_{1})\oplus
J_{1}(\lambda_{3})\oplus J_{1}(\lambda_{4})\oplus
J_{1}(\lambda_{5})$,

\item $2$ double roots $\lambda_{1},\lambda_{3}$  and $1$ simple
root $\lambda_{5}$ - $J_{2}(\lambda_{1})\oplus
J_{2}(\lambda_{3})\oplus J_{1}(\lambda_{5})$,

\item $1$ triple root $\lambda_{1}$  and $2$ simple roots
$\lambda_{4},\lambda_{5}$ - $J_{3}(\lambda_{1})\oplus
J_{1}(\lambda_{4})\oplus J_{1}(\lambda_{5})$,

\item $1$ triple root $\lambda_{1}$  and $1$ double root
$\lambda_{4}$ - $J_{3}(\lambda_{1})\oplus J_{2}(\lambda_{4})$,

\item $1$ root of order $4$ $\lambda_{1}$  and $1$ single root
$\lambda_{5}$ - $J_{4}(\lambda_{1})\oplus J_{1}(\lambda_{5})$,

\item $1$ root of order $5$ $\lambda_{1}$ - $J_{5}(\lambda_{1})$.
\end{enumerate}

\end{enumerate}

%%%%%%%%%%%%%%%%%%%%%%%%%%%%%%%%%%%%%%%%%%%%%%%%%%%%%%%%%%%%%

\end{document}